\documentclass[12 pt]{amsart}

\usepackage{amsthm}
\usepackage{amssymb}
\usepackage{latexsym}
\usepackage{url}
\usepackage{graphicx} 
\usepackage{subfig}

\author{Felix Goldberg}
\address{Hamilton Institute, National University of Ireland Maynooth, Ireland}
\email{felix.goldberg@gmail.com}

\author{Abraham Berman}
\address{Department of Mathematics, Technion-IIT, Technion City, Haifa 32000, Israel}
\email{berman@tx.technion.ac.il}

\title{Zero forcing for sign patterns}
\date{June 25, 2013}

\dedicatory{Dedicated to the memory of Michael Neumann and Uriel Rothblum}

\newtheorem{thm}{Theorem}[section]

\newtheorem{rmrk}[thm]{Remark}

\newtheorem{defin}[thm]{Definition}
\newtheorem{expl}[thm]{Example}

\newtheorem{qstn}[thm]{Question}

\newtheorem{rul}{Rule}

\begin{document}

\begin{abstract}
We introduce a new variant of zero forcing - \emph{signed zero forcing}. The classical zero forcing number provides an upper bound on the maximum nullity of a matrix with a given graph (i.e. zero-nonzero pattern). Our new variant provides an analogous bound for the maximum nullity of a matrix with a given sign pattern. This allows us to compute, for instance, the maximum nullity of a $Z$-matrix whose graph is $L(K_{n})$, the line graph of a clique.
%Previously, the zero forcing technique 
%It provides a better upper bound on the maximum nullity of symmetric $Z$-matrices than the one obtained by classical zero forcing.
\end{abstract}



\maketitle

\section{Introduction}
One of the most vibrant areas in algebraic graph theory in recent years has been the study of minimum rank/maximum nullity problems. If
$G$ is a simple graph on $n$ vertices, labelled $\{1,\ldots,n\}$, let $M^{\mathbb{F}}(G)$ be the maximum possible nullity of a symmetric $n \times n$ matrix over field $\mathbb{F}$ whose graph is $G$ (\emph{i.e.} for all $i \neq j$, the $ij$ entry of the matrix is non-zero if and only if $i,j$ are adjacent in $G$). The minimum possible rank of such a matrix is denoted $mr^{\mathbb{F}}(G)$ and it is not hard to see that $M^{\mathbb{F}}(G)+mr^{\mathbb{F}}(G)=n$. Therefore, lower bounds on $mr$ are equivalent to upper bounds on $M$. For many such results obtained by various authors to date we can direct the reader to the surveys \cite{FalHog07, FalHog11} by Fallat and Hogben.

A popular way of bounding the maximum nullity from above is by the so-called \emph{zero forcing} method. Zero forcing is a combinatorial game played on the vertices of a graph, during which the vertices are colored black and white. It was first formally defined and studied in \cite{AIM_zeroforce08}, although inchoate versions had been in use for some time before. Further variants of the basic idea were introduced and are surveyed in \cite{NewZF_with_TW}. 
%for example \emph{positive semidefinite zero forcing} \cite{Fallat_etal_minrank_zf}. 

%%We describe the classical zero forcing in Section \ref{sec:classical}. Its relevance to maximum nullity lies in the following result:
%%\begin{thm}\cite{AIM_zeroforce08}\label{thm:classical}
%%$M^{\mathbb{F}}(G) \leq Z(G)$, for any field $\mathbb{F}$.
%%\end{thm}

In this paper we address an issue that was raised by Hogben in \cite[p. 208]{Hog11}: how to adapt the zero forcing number argument to sign patterns? 

A sign pattern $P$ is a $m \times n$ matrix whose entries are drawn from $\{+,-,0,?\}$ and a real $m \times n$ matrix $A$ is said to have sign pattern $P$ if the sign of $A_{ij}$ is equal to $P_{ij}$ whenever $P_{ij} \neq ?$. There is no restriction on the $ij$th entry of $A$ if $P_{ij}=?$. 
\begin{rmrk}
All sign patterns in this paper are assumed to be square, \emph{i.e.} $m=n$.
\end{rmrk}
A square sign pattern naturally corresponds to a signed directed graph and the indices in $\{1,2,\ldots,n\}$ will be referred to as the vertices of the sign pattern.
\begin{defin}
Let $P$ be a square sign pattern. If $P_{ij} \neq ?$ whenever $i \neq j$ we will say that $P$ \emph{has fixed periphery}.
\end{defin}
\begin{rmrk}
All sign patterns in this paper are assumed to have fixed periphery.
\end{rmrk}

We shall denote by $M^{\mathbb{R}}(P)$ or simply $M(P)$ the maximum nullity of a real matrix whose sign pattern is $P$. 
%We shall also speak of a vertex of a sign pattern, meaning an index of $\{1,2,\ldots,n\}$.

The problem of determining the ranks/nullities of sign patterns has been first posed by Johnson \cite{Joh82} in 1982, with the last decade or so witnessing a vigorous renewal of interest in it, motivated partly by applications in complexity theory. For further background information we can direct the reader to the introduction of \cite{Li_etal12} where a list of most of the relevant papers on the subject to date has been compiled.
 
%There is a naive way to bound $M^{\mathbb{R}}(P)$ by a zero forcing argument. The pattern $P$ corresponds to a directed signed graph $G_{s}$; we can convert $G_{s}$ into a simple graph $G_{0}$ by erasing all signs and arrows (and suppressing multiple edges). Clearly the graph of any matrix corresponding to $P$ is $G_{0}$. Thus we have:
%$$
%M^{\mathbb{R}}(P) \leq M^{\mathbb{R}}(G_{0}) \leq Z(G_{0}).
%$$ 

We introduce a new variant, \emph{signed zero forcing}, which allows us in some cases to give non-trivial upper bounds on the maximum nullity of a sign pattern. Our main result is Theorem \ref{thm:main} which is the counterpart for sign patterns of the known upper bound on the maximum nullity given by classical zero forcing (we quote the latter bound here as Theorem \ref{thm:classical}).
 
The plan of the paper is: the classical game is described and analyzed in Section \ref{sec:classical} and our new game is defined in Section \ref{sec:newzf}. Section \ref{sec:proof} is devoted to the proof of Theorem \ref{thm:main} while in Section \ref{sec:branch} we very briefly sketch an extension of our new game that can yield even better bounds. Then, starting with Section \ref{sec:graphs} we turn our attention to graphs again and show how a corollary of Theorem \ref{thm:main} can be applied to bound the maximum possible nullity of a $Z$-matrix with a given graph.

\section{Classical zero forcing and nullity}\label{sec:classical}

We start the game by choosing a subset of vertices $S \subseteq V(G)$ and coloring them black. The other vertices remain white for now. Then the following color-change rule is applied until no more changes are possible. The resultant coloring is called the \emph{derived coloring of $S$}:

\begin{rul}[Classical zero forcing rule]\label{rul:zf_classical}
If $u$ is a black vertex of $G$ and exactly one neighbour $w$ of $u$ is white, then change the color of $w$ to black.
\end{rul}

A set $S \subseteq V(G)$ is said to be a \emph{zero forcing set for $G$} if all vertices of $G$ are black in the derived coloring of $S$. The \emph{zero forcing number of $G$}, $Z(G)$, is then the minimum size of a zero forcing set for $G$.

\begin{thm}\cite{AIM_zeroforce08}\label{thm:classical}
$M^{\mathbb{F}}(G) \leq Z(G)$, for any field $\mathbb{F}$.
\end{thm}

%\section{Classical zero forcing and nullity}\label{sec:classical}

%In order to supply the context for our analysis in Section \ref{sec:newzf} 
Let us briefly review here the proof of Theorem \ref{thm:classical} given in \cite[Propositions 2.2-2.4]{AIM_zeroforce08}. Some notations: we will denote the kernel of matrix $A$ by $\ker{A}$ and the subvector of vector $x$ formed on the entries indexed by a set $S$ by $x|_{S}$.

Consider a symmetric matrix $A$ whose graph is $G$ and a vector $x \in \mathbb{F}^{n}$. Denoting adjacency in $G$ by $\sim$, we have for any vertex $u$:

\begin{equation}\label{eq:basic}
(Ax)_{u}=a_{uu}x_{u}+\sum_{v \sim u}{a_{uv}x_{v}}.
\end{equation}

Now suppose that $S$ is a zero forcing set for $G$ and that $x \in \ker{A}$. We argue that if $x|_{S}=0$, then $x=0$. Indeed, there is some vertex $u \in S$ to which we can apply Rule \ref{rul:zf_classical} (or else, $S=V(G)$ and we are done). That is, $u$ has only one white neighbour, $w$. Looking at (\ref{eq:basic}), we obtain 
\begin{equation}\label{eq:one}
0=(Ax)_{u}=a_{uw}x_{w}
\end{equation}
and thus $x_{w}=0$. An inductive argument shows that in fact all entries of $x$ are zero.

Finally, suppose, for the sake of contradiction, that $M^{\mathbb{F}}(G)>Z(G)$. Therefore there is a symmetric matrix $A$ (whose graph is $G$), such that the nullity of $A$ is greater than $Z(G)$. Let $S$ be a minimum zero forcing set, with $|S|=Z(G)$. A standard dimensional argument now implies the existence of a nonzero $x \in \ker{A}$ that vanishes on the vertices of $S$. By the foregoing argument if follows that $x=0$ - a contradiction.

%To conclude this section, we remark that the proof of Theorem \ref{thm:psd} given in \cite{Fallat_etal_minrank_zf} proceeds along rather different lines.

\section{Signed zero forcing}\label{sec:newzf}

\subsection{Informal overview of the new method}\label{sec:newmot}

The main plank of the argument we discussed in the previous section was the deduction from Equation (\ref{eq:basic}) that some of the entries of the vector $x$ must be zero, given a set of known zero coordinates of $x$. To achieve that goal we had to reduce the expression on the right hand side of (\ref{eq:basic}) to just one summand in (\ref{eq:one}), since we had no information on the signs of the summands of \eqref{eq:basic}. 

However, now that we are no longer considering zero-nonzero patterns but sign patterns (with fixed periphery), the signs of the $a_{uv}$s are fixed and known and therefore the signs of the summands of \eqref{eq:basic} depend only on the signs of the $x_{v}$s. We shall now see how to make good use of this fact.

Let $A$ be a matrix whose sign pattern is $P$ and write $j \rightarrow i$ whenever $P_{ij} \in \{+,-\}$. Let $u \in \{1,2,\ldots,n\}$ and assume for the moment that $x_{u}=0$ as before. We denote $$W=\{w | w \rightarrow u\}.$$ Then Equation \eqref{eq:one} takes on the form
\begin{equation}\label{eq:h}
0=\sum_{w \in W}{a_{uw}x_{w}}.
\end{equation}
%$$ 

Define now also the sets 
$$W_{+}=\{w \in W| a_{uw}x_{w} \geq 0\},$$ 
$$W_{-}=\{w \in W| a_{uw}x_{w} \leq 0\}.$$ 

If, say, $W=W_{+}$ then the only way for Equation \eqref{eq:h} to hold is to have $x_{w}=0$ for all $w \in W$. This inference corresponds to the blackening of vertices in the zero forcing game - we deduce that some of the entries of the vector $x$ must be zero and blacken the corresponding vertices.

Another possible deduction from Equation \eqref{eq:h} is in the case that $W=W_{+} \cup \{w\}$. In this case we immediately deduce that $a_{uw}x_{w} \leq 0$ must be true for \eqref{eq:h} to hold and thus the weak sign of $x_{w}$ is the opposite of the sign of $a_{uw}$.
%, that is $P_{uw}$. 
Thus we see that knowledge of the signs of some entries of $x$ can lead to the determination of the signs of other entries. (the case $W=W_{-} \cup \{w\}$
is, of course, symmetric).

What we need now is a way to incorporate this reasoning into the zero forcing game. To do this, we add to the game another element: white vertices may now be marked as $+$ or $-$ vertices. \emph{The markers at a given moment represent the information we have at this stage about the possible weak signs of the entries of $x$}. As the game progresses, marked white vertices can be colored black (when we learn that their entries are in fact zero) - in which case the sign marker disappears under the black paint.

%\begin{rmrk}
%The $+$ and $-$ markers encode the information we currently possess about the possible signs of the different entries in $x$.  
%\end{rmrk}

Finally, we observe that the assumption $x_{u}=0$ is too stringent. In fact, if $P_{uu} \neq ?$ we can include the term $a_{uu}x_{u}$ in \eqref{eq:h} and analyze it on the same basis as the other terms. In combinatorial terms, this means that our new color-change rule can be applied either to black vertices or to white vertices whose "loop" has a known sign.

%\begin{rmrk}
%Notice that unlike in Section \ref{sec:classical} we did not have to assume that $x_{u}=0$. However, we did implicitly assume that if $x_{u} \neq 0$ then the sign of $x_{u}$ is known. In other words that $P_{uu} \neq ?$.
%\end{rmrk}

\subsection{Seeding markers}\label{sec:seeding}
We have seen in the previous subsection that given the signs of some entries in $x$ (which correspond to  vertices of the sign pattern) we may be able to determine the signs of some other entries. But how does one, so to speak, gets the ball rolling? How does one obtain the first sign?

This is actually very simple: after the vertices of $S$ have been blackened, we can choose \emph{any} vertex and mark it with $+$! 

Since the sign markers represent our belief about the signs of the entries of $x$, our seemingly arbitrary choice of $+$ amounts to choosing one vector from the pair $x,(-x)$. We can then continue the game, using the chosen vector, obtainining the same derived coloring for $x$ and $-x$. 
%(in the sense that the same vertices will end up being black, whether we choose $x$ or $(-x)$).
%Since by assumption no vertices are marked at the time of seeding, it follows that we are not losing generality by our assignment.
	
One important observation remains to be made: we may seed the markers anew at some point in the game. This happens if all marked vertices had been blacked out, but some white vertices still remain.

\subsection{Formalization}

The \emph{signed zero forcing} game proceeds as the usual zero forcing game, but a new zero forcing rule is applied. Note that clause (a) of the new rule subsumes the classical rule and that clause (d) is a formalization of the seeding concept discussed above. 

To state the rule formally a few extra notations are needed: the sign inversion function $\iota$ is defined by $\iota(+)=-$ and $\iota(-)=+$; the letters $s$ and $t$ will be used as indices taking values in $\{+,-\}$ and $s \cdot t$ will be defined as $+$ if $s=t$ and as $-$ if $s \neq t$. Finally, if $w$ is a white vertex, we will denote by $m(w)$ its marker if it is marked; otherwise we will write $m(w)=*$.

\begin{rul}[Signed zero forcing rule for sign pattern $P$]\label{rul:new}
Let $u$ be a vertex of $P$ such that either $u$ is black or $u$ is white and $P_{uu} \neq ?$.
Let $$W=\{w | w \ \textit{is white} \wedge w \rightarrow u \}\cup \{u\}.$$  
Define $$W_{+}=\{w \in W| m(w)=P_{uw}\}, W_{-}=\{w \in W| m(w) \neq P_{uw}\}$$ and $$W_{*}=\{w \in W| m(w)=*\}.$$
\begin{enumerate}
\item
If $W=\{w\}$, color $w$ black.
\item
If either $W_{+}=W$ or $W_{-}=W$, color all vertices in $W$ black.
\item
If $W_{s} \neq \emptyset,W_{\iota(s)}=\emptyset$, and $W_{*}=\{w\}$, mark $w$ with $P_{uw} \cdot \iota(s)$.
\item 
If no white vertices in the whole graph are marked and $u$ is white, then mark $u$ with $+$.
\end{enumerate}
\end{rul}

\begin{defin}
The \emph{signed zero forcing number} of $P$, $Z_{\pm}(P)$, is the size of the minimum forcing set when playing the game as outlined in this section.
\end{defin}

The arguments expounded in Sections \ref{sec:newmot} and \ref{sec:seeding} can be crystallized into the following statement, which is the main result of the paper. 

\begin{thm}\label{thm:main}
Let $P$ be a sign pattern with fixed periphery. Then:
$$M^{\mathbb{R}}(P) \leq Z_{\pm}(P).$$
\end{thm}

%The proof of the theorem, presented in the next section, is but a technical affair of carefully coding the informal arguments of Section \ref{sec:newzf}. Therefore, we recommend to the reader to proceed directly to Section \ref{sec:q3} where we demonstrate the operation of the signed zero forcing game on an example.

Before proceeding to the formal proof of Theorem \ref{thm:main} (in effect, just a careful codification of the informal arguments given before) we wish to illustrate the operation of the game on a small example. 
%Further examples will be presented in subsequent sections.

\begin{expl}\label{ex:hadamard}
%The signed zero forcing number $Z_{\pm}(P)$ is defined then in the obvious manner, as the cardinality of the smallest set $S$ whose derived coloring using Rule \ref{rul:newen} is all black.

%As an illustration, let us now show how the game with Rule \ref{rul:newen} proceeds for 
Consider the $4 \times 4$ Hadamard sign pattern:
$$
P=\left(\begin{array}{cccc} 
+&+&+&+\\ 
+&-&+&-\\
+&+&-&-\\
+&-&-&+
\end{array}\right).
$$

We are going to show that $Z_{\pm}(P)=2$, and therefore $M(P) \leq 2$. In fact, Hogben \cite[p.206]{Hog11} has shown that $mr(P)=3$, and therefore $M(P)=1$; we will see in Section \ref{sec:branch} how a natural extension of our technique enables us to prove that $M(P)=1$ just as well.
%Since this pattern is full, our method cannot be expected to produce optimal results on it (and indeed, Hogben \cite[p.206]{Hog11} has shown that $mr(P)=3$, and therefore $M(P)=1$). However, it is small enough to serve as a convenient didactic device to illustrate the operation of the new rule.

Let $S=\{1,2\}$ and colour the vertices of $S$ black. Now we apply Rule \ref{rul:new}(d) to mark vertex $3$ with $+$. Next we apply Rule \ref{rul:new}(c) to $u=3$. In this case we have $W=\{3,4\}$ and $W_{-}=\{3\},W_{*}=\{4\}$. We see that vertex $4$ can be marked with $P_{34} \cdot \iota(-)$, that is with $-$. Finally, we apply Rule \ref{rul:new}(b) to vertex $u=2$; now we have $W=W_{+}=\{3,4\}$ and so both $3$ and $4$ can be blackened, finishing the game.

Note that we could have replaced the last move with the application of Rule \ref{rul:new}(b) to $u=4$ instead, in which case we would have had $W=W_{-}=\{3,4\}$.

\end{expl}

\section{Proof of Theorem \ref{thm:main}}\label{sec:proof}

\begin{proof}
Suppose that $S$ is a minimum forcing set under Rule \ref{rul:new}, so that $|S|=Z_{\pm}(P)$. This means that the signed zero forcing game can be played in, say, $m$ moves: $\mathbb{M}_{1},\mathbb{M}_{2},\ldots,\mathbb{M}_{m}$. The first move $\mathbb{M}_{1}$ consists of coloring the vertices of $S$ in black. The remaining moves are repeated applications of Rule \ref{rul:new} - until all the vertices of $P$ are black.

Let $B^{k}$ be the sets of vertices that are black after the move $\mathbb{M}_{k}$ has been played. Observe that $B^{1}=S$ and $B^{m}=\{1,2,\ldots,n\}$. Let also $M_{k}$ be the set of white vertices with markers after move $\mathbb{M}_{k}$ and let $m_{k}(w)$ be the marker of $w$ at this stage, for a vertex $w \in M_{k}$.

Now let $A$ be a matrix whose sign pattern is $P$ and let $x \in \ker{A}$.

\emph{Claim} If $x|_{S}=0$, then for any $1 \leq k \leq m$ it holds that:
\begin{itemize}
\item
$x|_{B^{k}}=0$.
\item
If $w_{1},w_{2}\in M_{k}$ and $m(w_{1})=m(w_{2})$, then $x_{w_{1}}x_{w_{2}} \geq 0$.
\item
If $w_{1},w_{2}\in M_{k}$ and $m(w_{1}) \neq m(w_{2})$, then $x_{w_{1}}x_{w_{2}} \leq 0$.
\end{itemize}

The claim, once established, will show that for $x \in \ker{A}$, $x|_{S}=0$ entails $x=0$. The dimensional argument mentioned in Section \ref{sec:classical} will then finish the proof.

We proceed to prove the claim by induction on $k$. The claim is true for $k=1$ as $B^{1}=S$ and no vertices are marked at that stage. Suppose that the claim is true for some $k \geq 1$ and let us show it for $k+1$. The move $\mathbb{M}_{k+1}$ could have been the application of each of the four clauses of Rule \ref{rul:new} to a vertex $u \in B^{k}$. 

Recall now the basic formula (\ref{eq:h}):
\begin{displaymath}
%(Ax)_{u}=a_{uu}x_{u}+\sum_{v \sim u}{a_{uv}x_{v}}.
0=\sum_{w \in W}{a_{uw}x_{w}}.
\end{displaymath}

If $\mathbb{M}_{k+1}$ applied the first clause of Rule \ref{rul:new}, with $w$ being the sole white neighbour of $u$, then (\ref{eq:h}) simplifies to $0=a_{uw}x_{w}$ and therefore $x_{w}=0$. Since in this case 
$B^{k+1}=B^{k} \cup \{w\}$ and $M_{k+1} \subseteq M_{k}$, the claim is upheld.

Now suppose that $\mathbb{M}_{k+1}$ applied the second clause of Rule \ref{rul:new} to $u$, with, say, $W_{+}=W$. Then \eqref{eq:h} reduces to $$0=\sum_{v \in W^{+}}{a_{uv}x_{v}}$$
which can be rewritten as:
\begin{equation}\label{eq:cl2}
0=\overbrace{\sum_{\substack{v \in W^{+} \\ m(v)=+}}{a_{uv}x_{v}}}^{A} + 
\overbrace{\sum_{\substack{v \in W^{+} \\ m(v)=-}}{a_{uv}x_{v}}}^{B}.
\end{equation}
We claim that all the summands on the right hand side of \eqref{eq:cl2} have the same weak sign.
Indeed, all the $x_{v}$s that appear in $A$ have the same weak sign by the induction assumption and by the definition of $W^{+}$ the corresponding $a_{uv}$s are positive. On the other hand, by the induction assumption again, the weak sign of all the $x_{v}$s in $B$ is opposite to the weak sign of the $x_{v}$s in $A$, while the $a_{uv}$s in $B$ are negative (again by the definition of $W^{+}$). Summing up, we see that all summands, whether in $A$ or in $B$, have the same weak sign. However, this means that they must all be zero, as their sum is zero. Therefore $x_{v}=0$ for all $v \in W$ and the claim is upheld as $B^{k+1}=B^{k} \cup W$ and $M_{k+1} \subseteq M_{k}$.

Next we have to show that the claim remains valid when the the third clause of Rule \ref{rul:new} is applied. We shall only prove it for $s=+$ since the proof for $s=-$ is the same, \emph{mutatis mutandis}. 

In this case Equation (\ref{eq:h}) takes the form 
$$0=\overbrace{\sum_{v \in W^{+}}{a_{uv}x_{v}}}^{C} + a_{uw}x_{w},$$ where $w$ is the sole unmarked vertex in $W$. The same argument as before shows that all summands in $C$ have the same weak sign. Therefore, we deduce that the $a_{uw}x_{w}$ term has the opposite weak sign. If $v$ is some vertex in $W^{+}$ then $a_{uv}x_{v}$ and $a_{uw}x_{w}$ must have opposite weak signs. We now consider four possible cases and verify that in each of them the claim holds. Note that $m_{k+1}(w)=P_{uw} \cdot \iota(s)$; since $s=+$ it means in effect that $w$ is marked with $+$ if $a_{uw}>0$ and marked with $-$ if $a_{uw}<0$.

Case 1: $m_{k}(v)=+,m_{k+1}(w)=+$. Since $m_{k}(v)=+$, we see that $a_{uv}>0$. Also, $a_{uw}<0$, as discussed above. Therefore, for $a_{uv}x_{v}$ and $a_{uw}x_{w}$ to have opposite weak signs, $x_{v}$ and $x_{w}$ must have the same weak sign - which is just what the claim asserts.

Case 2: $m_{k}(v)=+,m_{k+1}(w)=-$. In this case, $a_{uv}>0$ and $a_{uw}>0$, so if $a_{uv}x_{v}$ and $a_{uw}x_{w}$ are to have opposite weak signs, then $x_{v}$ and $x_{w}$ must have opposite weak signs as well.

Case 3: $m_{k}(v)=-,m_{k+1}(w)=+$. Same kind of argument.

Case 4: $m_{k}(v)=-,m_{k+1}(w)=-$. Same kind of argument.

Finally, for the fourth clause of Rule \ref{rul:new} the claim is true in a trivial way.
\end{proof}

\section{Interlude - branching}\label{sec:branch}
Part of the charm of both classical and signed zero forcing games is that they proceed in a straightforward way (and so are easy to program). However, it might be possible to do better at the signed zero forcing game, at the price of introducing a branching element into the game. 

%The idea is best introduced by an example

To expound this idea, consider again Example \ref{ex:hadamard}. Colour the vertex $1$ in black and mark vertex $2$ with $+$, in accordance with Rule \ref{rul:new}(d). Now we create a branch in the game by considering three options for vertex $3$: marked with $+$, marked with $-$, or black. This corresponds to the three possible options for $x_{3}$: positive, negative, or zero. Let us consider all three options and see that they all lead to the blackening of all four vertices, \emph{viz.} to the conclusion that $x=0$.

Case 1: $m(3)=+$. Let $u=1$ - then we have $W_{+}=\{2,3\}$ and $W_{*}=\{4\}$. Therefore we can mark $4$ with $P_{14} \cdot \iota(+)= + \cdot -=-$. Now let $u=4$ and we obtain for it $W_{-}=\{2,3,4\}$. This means we can blacken all three white vertices!

Case 2: $m(3)=-$. Let $u=2$ and then $W_{-}=\{2,3\}$ and $W_{*}=\{4\}$. Therefore we get to mark vertex $4$ with $P_{24} \cdot \iota(-)=- \cdot +=-$. But now let $u=3$ and we see that $W_{+}=\{2,3,4\}$. Once again all three white vertices are blackened.

Case 3: $3$ is black. First consider $u=1$, leading to marking vertex $4$ with $-$. But now if we look at $u=3$ we see that $W_{+}=\{2,4\}$. Total blackout once again.
 
What have introduced here, in effect, a new variant of the game. A tentative name for it can be \emph{branched signed zero forcing}. If we denote it by $Z_{\pm}^{b}$ then we have shown that in our example $Z_{\pm}^{b}(P)=1$ and therefore $M(P) \leq Z_{\pm}^{b}(P) =1$.

It would be interesting to find more cases in which $Z_{\pm}^{b}<Z_{\pm}$.

\section{Signed zero forcing for graphs}\label{sec:graphs}
%\subsection{$\mathcal{Z}$-matrices}
%We now shift our emphasis to the case of patterns all of whose off-diagonal entries are either $0$ or $s$ for a fixed $s \in \{+,-\}$. On the other hand, we will not restrict the diagonal entries. This leads us to consider, in effect, the following matrix family:
\begin{defin}
Let $\mathcal{Z}$ be a family of real symmetric matrices defined in the following way: $A \in \mathcal{Z}$ if and only if all non-zero off-diagonal entries of $A$ share the same weak sign.
\end{defin}

The class $\mathcal{Z}$ includes some important matrix classes, such as symmetric $Z$-matrices (in particular, Stieltjes matrices) and symmetric entrywise nonnegative matrices. These leads us to take interest in the following graph parameter: 
%Another special case of symmetric $Z$-matrices is proferred by the \emph{Colin de Verdi\'{e}re matrices} and the present work is an outgrowth of \cite[Chapter 5]{Mythesis} where an early version of signed zero forcing was used to bound from above the Colin de Verdi\'{e}re number.

\begin{defin}
The maximum possible nullity of a matrix in $\mathcal{Z}$ whose graph is $G$ will be denoted $M^{\mathbb{R}}_{\mathcal{Z}}(G)$.
\end{defin}

Clearly, we can bound $M^{\mathbb{R}}_{\mathcal{Z}}(G)$ from above by the following inequalities:
$$
M^{\mathbb{R}}_{\mathcal{Z}}(G) \leq  M^{\mathbb{R}}(G) \leq Z(G).
$$

%We  order to provide an upper bound on $M^{\mathbb{R}}_{\mathcal{Z}}(G)$

But now we are in a position to do better. Define a sign pattern $P_{\mathcal{Z}}(G)$ in the following way: 
$$
P_{\mathcal{Z}}(G)_{ij}=\begin{cases}
? & \text{, if } i=j \\
-   & \text{, if } i \sim j \\
0 & \text{, if } i \neq j, i \not\sim j. \\
\end{cases}
$$

The following result is then an immediate consequence of Theorem \ref{thm:main}:
%\begin{defin}

%\end{defin}

%\subsection{}
%The signed zero forcing game is then played for a graph in much the same way as for a sign pattern, with the single difference that we restrict once again the zero forcing rule to black vertices (as the signs of the diagonal entries $a_{uu}$ are unspecified, we cannot deduce anything about equations in which an $a_{uu}$ appears coupled to a nonzero $x_{u}$). For clarity we write out the appropriate zero forcing rule for graphs:
%%\begin{rul}[Signed zero forcing rule for graph $G$]\label{rul:new}
%%Let $u$ be a black vertex of $G$ and let $W$ be the set of white neighbours of $u$. Some of them are marked with $+$, some with $-$, and some are unmarked. Call these subsets $W_{+},W_{-}$, and $W_{*}$, respectively.
%%\begin{enumerate}
%%\item
%%If $W=\{w\}$, color $w$ black.
%%\item
%%If either $W_{+}=W$ or $W_{-}=W$, color all vertices in $W$ black.
%%\item
%%If $W_{s} \neq \emptyset,W_{\iota(s)}=\emptyset$, and $W_{*}=\{w\}$, mark $w$ with $P_{uw} \cdot \iota(s)$.
%%\item 
%%If no white vertices in the whole graph are marked then mark some white vertex with $+$.
%%\end{enumerate}
%%\end{rul}

%%\begin{defin}
%%The \emph{signed zero forcing numbe r} of $
%%G$, $Z_{\pm}(G)$, is the size of the minimum forcing set when playing the game with Rule \ref{rul:new}.
%%\end{defin}

%Theorem \ref{thm:main} has a natural counterpart for graphs, whose proof is essentially the same:
\begin{thm}\label{thm:maing}
%$M^{\mathbb{R}}_{\mathcal{Z}}(G) \leq Z_{\pm}(G)$.
$M^{\mathbb{R}}_{\mathcal{Z}}(G) \leq Z_{\pm}(P_{\mathcal{Z}}(G))$.
\end{thm}

We will henceforth abuse notation and refer to $Z_{\pm}(P_{\mathcal{Z}}(G))$ as $Z_{\pm}(G)$.

%\subsection{Example - the hypercube $Q_{3}$}\label{sec:q3}
\begin{expl}\label{ex:q3}

Let $Q_{3}$ be the hypercube of dimension $3$, having eight vertices, as shown in Figure \ref{fig:cube3}.

It is not difficult to verify directly that $Z(Q_{3})=4$. We are going to show that $Z_{\pm}(Q_{3})=3$. Consider the drawing in Figure \ref{fig:cube3} on the left. We take $S=\{1,3,7\}$. We seed vertex $5$ with $+$ and then apply Rule \ref{rul:new}(c) to $7$, allowing us to mark $8$ with $-$.
%Let us seed the markers: vertex $7$ has exactly two white neighbours: the vertices $5$ and $8$. So we mark the one with $+$ and the other with $-$, according to Rule \ref{rul:new}(c). 
The state of the game is shown in Figure \ref{fig:cube3} on the right.

\begin{figure}[h]
\begin{center}$
\begin{array}{cc}
\includegraphics[width=2.5in]{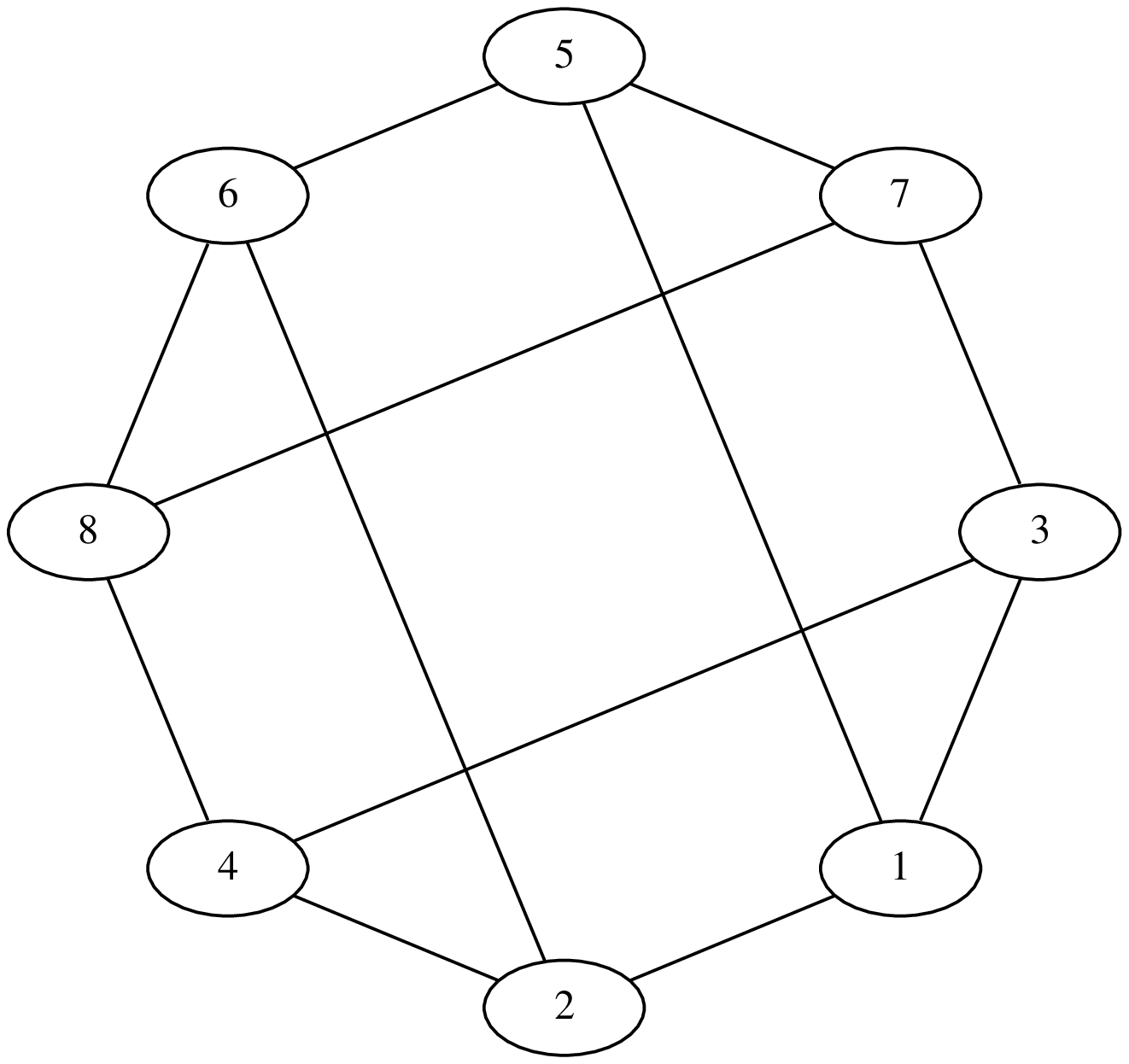} &
\includegraphics[width=2.5in]{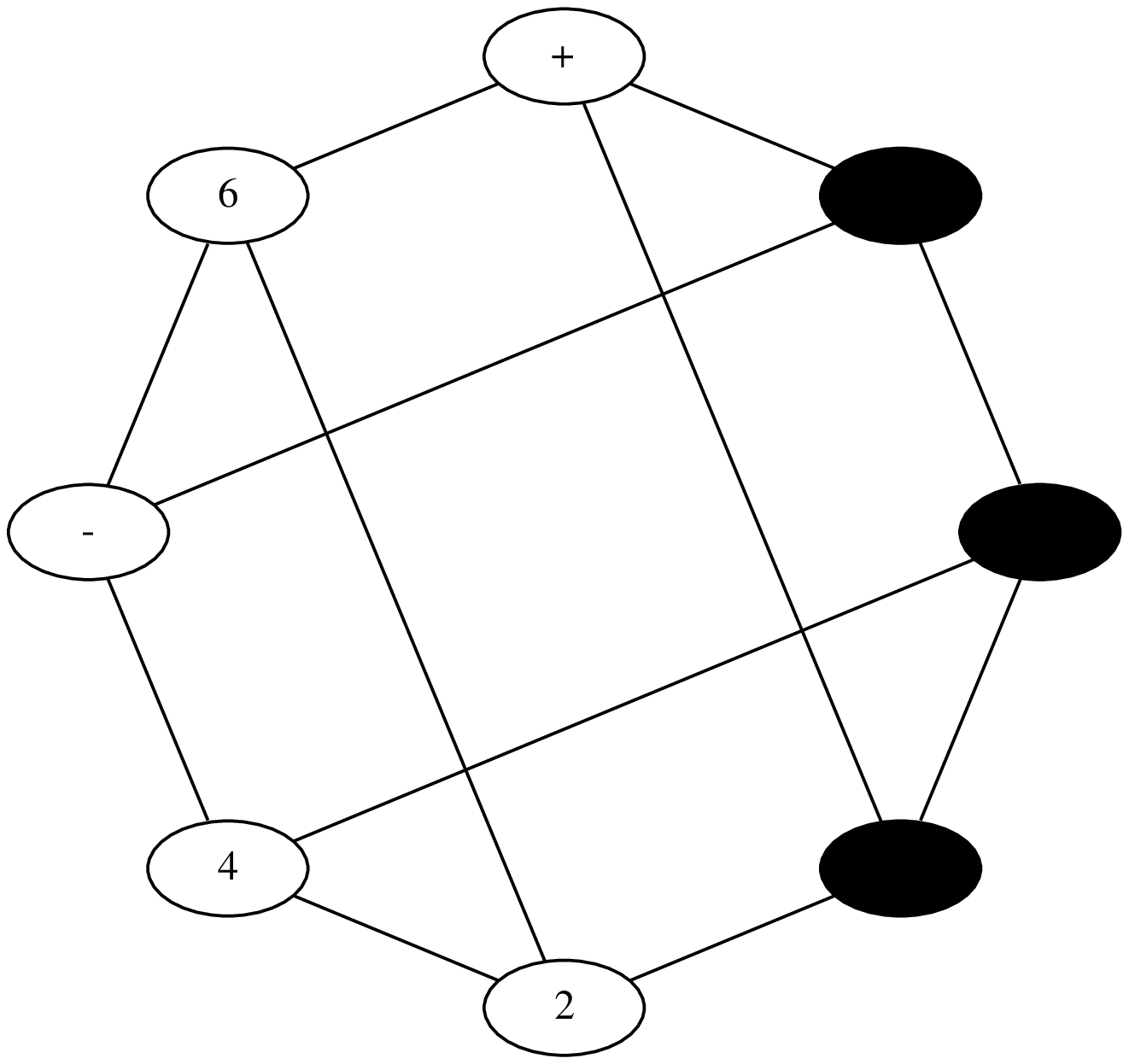}
\end{array}$
\end{center}
\caption{Left - $Q_{3}$, Right - $Q_{3}$ after the first stage of the game}\label{fig:cube3}
\end{figure}

Our next step is to apply Rule \ref{rul:new}(c) to vertex $1$, allowing us to mark vertex $2$ with $-$. We also apply Rule \ref{rul:new}(a) to vertex $3$, allowing us to colour the vertex $4$ black. The state of the game now is pictured in Figure \ref{fig:state2} on the left. 

Now comes the decisive blow - we apply Rule \ref{rul:new}(b) to vertex $4$ and blacken the vertices $2$ and $8$. The resultant state, depicted in Figure \ref{fig:state2} on the right, is such that six vertices out of eight are already black and the remaining two are dispatched easily using Rule \ref{rul:new}(a).

Finally, it is not hard to see that if we start with two vertices,  it is impossible to blacken all vertices. Therefore $Z_{\pm}(Q_{3})=3$.

\begin{figure}[h]
\begin{center}$
\begin{array}{cc}
\includegraphics[width=2.5in]{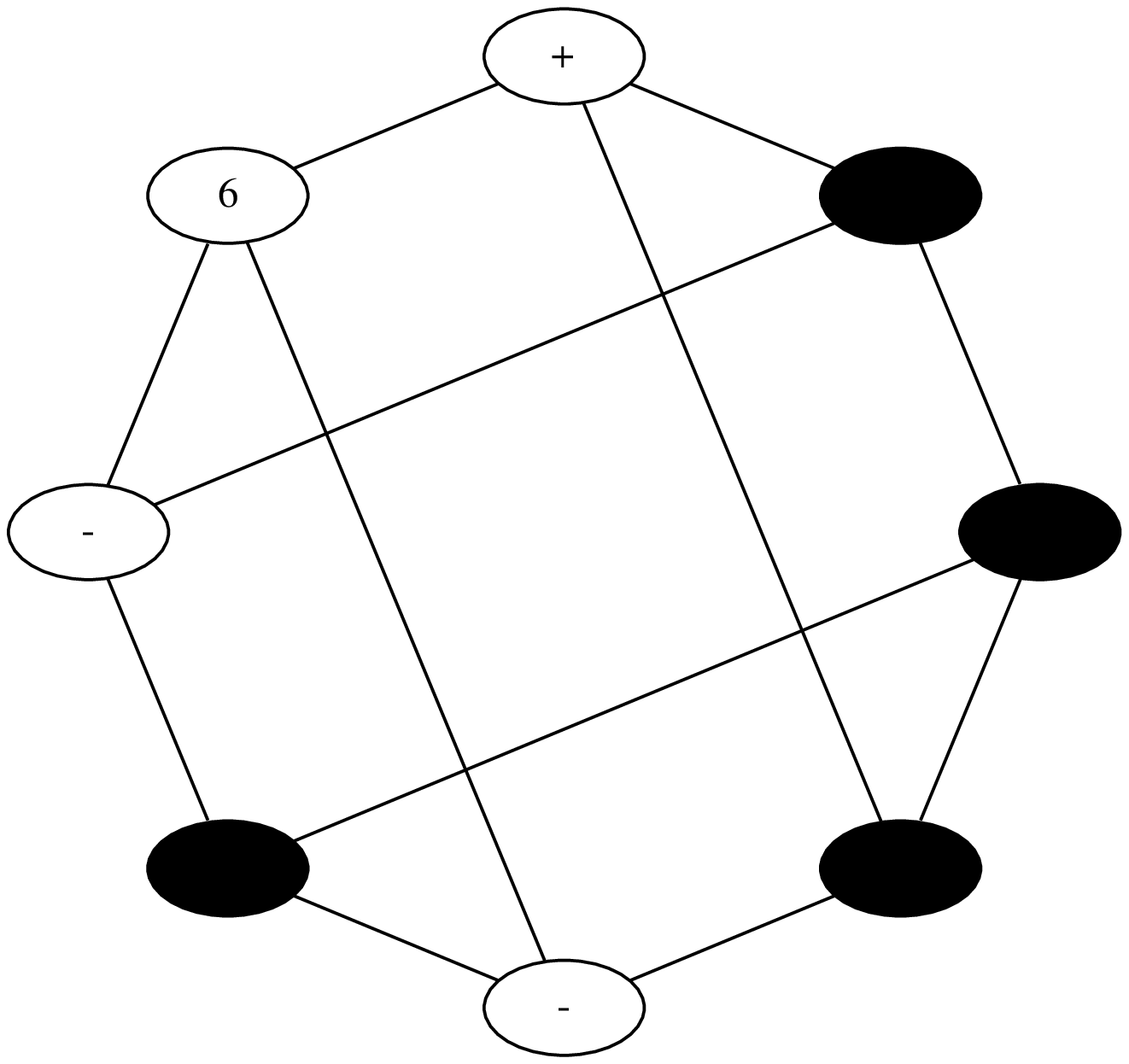} &
\includegraphics[width=2.5in]{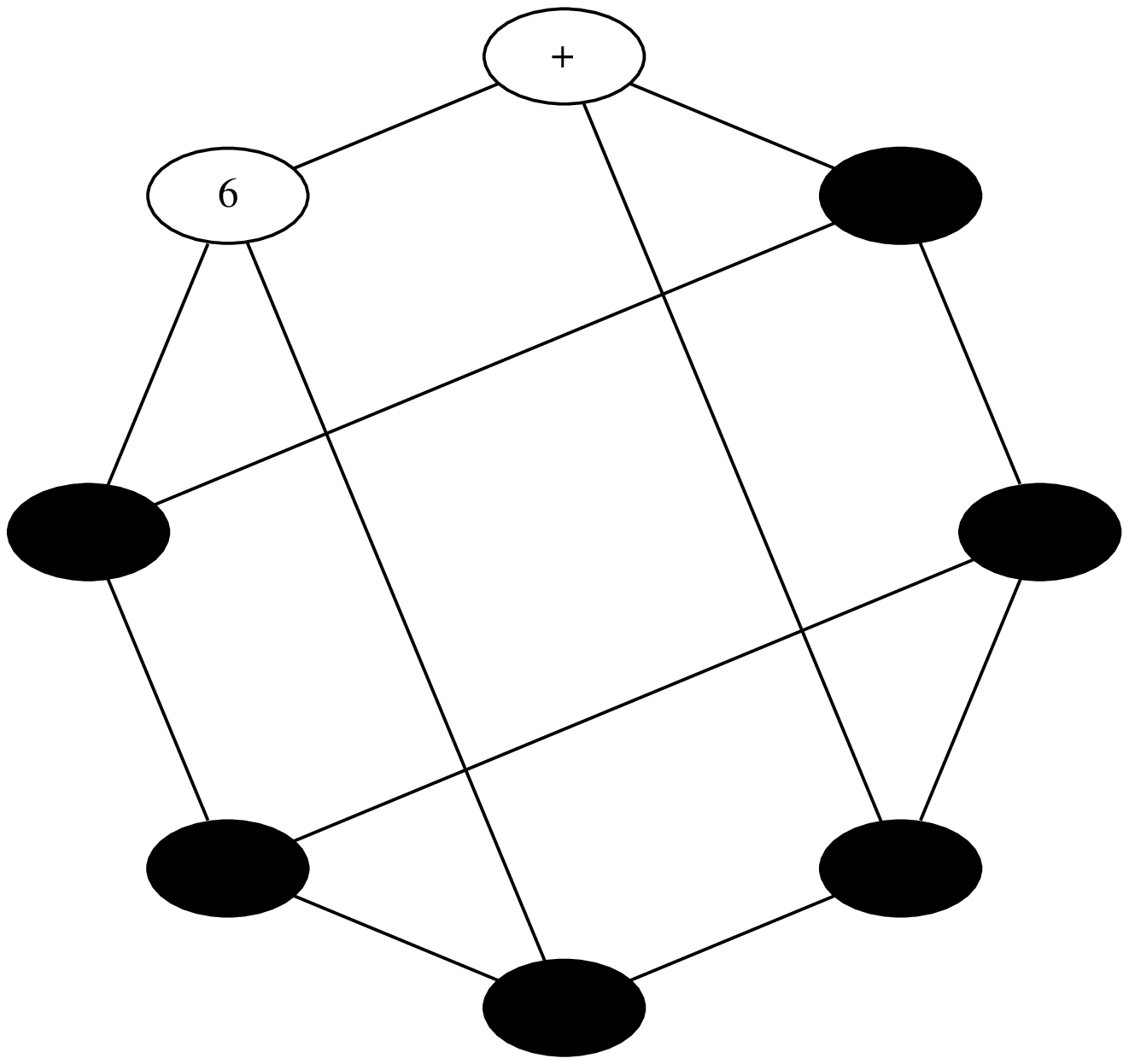}
\end{array}$
\end{center}
\caption{}\label{fig:state2}
\end{figure}

\end{expl}

We can use Example \ref{ex:q3} to give an upper bound on the maximum $\mathcal{Z}$-multiplicity of a hypercube of higher dimensions.

%\subsection{General hypercubes}
%The classical zero forcing number $Z(G)$ and its positive semidefinite counterpart $Z_{+}(G)$ of $Q_{d}$ both equal $2^{d-1}$:
%Huang, Chang and Yeh have recently obtained the following result:
For general hypercubes the following result is known:
\begin{thm}\cite{HuaChaYeh10} $M^{\mathbb{F}}(Q_{d})=Z(Q_{d})=2^{d-1}$, for any field $\mathbb{F}$.
\end{thm}

%\begin{thm}\cite[Theorem 3.1.8]{Peters_thesis} $M_{+}(Q_{d})=Z_{+}(Q_{d})=2^{d-1}$.
%\end{thm}

The argument for the upper bound $Z(Q_{d}) \leq 2^{d-1}$ had been given in \cite[Proposition 2.5]{AIM_zeroforce08} where it was shown that if $G \square H$ is a Cartesian graph product and $S$ is a classical zero forcing set for $G$, then by taking a copy of $S$ in each fiber of $G$ inside $G \square H$, we obtain a classical zero forcing set for $G \square H$.

%It can be shown in a straightforward way that 
\emph{Mutatis mutandis} the same argument works for the signed zero forcing game. 
%On the other hand, it is not difficult to see that $Z_{\pm}(Q_{3})=3$. 
Therefore, we have:

\begin{thm}
For all $d \geq 3$, $M^{\mathbb{R}}_{\mathcal{Z}}(Q_{d}) \leq Z_{\pm}(Q_{d}) \leq 3 \cdot 2^{d-3}$.
\end{thm}

%\section{Computation of $Z$ and $Z_{\pm}$}

%\section{The smallest graphs for which $Z_{\pm}<Z$}
\section{Graphs with $Z_{\pm}(G)<Z(G)$}
%The hypercube $Q_{3}$ is an example of a graph for which $Z_{\pm}$ is strictly less than $Z$ and it has eight vertices.
Since Rule \ref{rul:new}(a) is equivalent to Rule \ref{rul:zf_classical}, it is always true that $Z_{\pm}(G) \leq Z(G)$. As we have seen for hypercubes, strict inequality may obtain, in which case $Z_{\pm}$ provides substantially new information over that already given by $Z$. In this section we will show another family of graphs for which $Z_{\pm} <Z$ and also report the result of a computer search over small graphs.
%We do not have yet a characterization of the cases in which strict inequality obtains but in this section we will offer some initial results pertaining to this question.

But our first result is a negative one - for trees the two parameters coincide:
\begin{thm}\label{thm:trees}
Let $T$ be a tree. Then $Z_{\pm}(T)=Z(T)$.
\end{thm}
\begin{proof}
It is a well-known fact (cf. \cite[Lemma 1.2]{Alba_etal06}) that any symmetric sign pattern is congruent via a positive diagonal matrix to a pattern all of whose nonzero off-diagonal entries are $+$. This implies that $M_{\mathcal{Z}}(T)=M(T)$. It is also known \cite[Proposition 4.2]{AIM_zeroforce08} that $M(T)=Z(T)$ for any tree $T$. 

Therefore we can write:
$$M_{\mathcal{Z}}(T) \leq Z_{\pm}(T) \leq M(T)=M_{\mathcal{Z}}(T)$$
and equality must hold throughout.
\end{proof}

\subsection{Small graphs}
In order to find the smallest order of a graph with this property we ran a computer search over the catalogue of small graphs made publicly available by Brendan McKay \cite{McKayGraphsData}. We found out that the smallest order is $6$ and that there are exactly two graphs with $Z_{\pm}<Z$ on six vertices; for both of them $Z_{\pm}=3,Z=4$. They are shown in Figure \ref{fig:min}.

One of the referees has pointed out that both these graphs have also the properties $M=4$ and $M_{\mathcal{Z}}=3$. While $M=4$ follows for both from \cite[Proposition 4.3]{AIM_zeroforce08}, to see that $M_{\mathcal{Z}}=3$ for, say, the graph on the right in Figure \ref{fig:min}, the referee has suggested the following argument: a matrix from $\mathcal{Z}$ associated to it must be of the form
$$
\left(\begin{array}{ccccccc} 
d_{1}&0&a&b&c&d\\ 
0&d_{2}&s&t&u&v\\
a&s&d_{3}&0&w&x\\
b&t&0&d_{4}&y&z\\
c&u&w&y&d_{5}&0\\
d&v&x&z&0&d_{6}
\end{array}\right),
$$
with all the off-diagonal variables being strictly positive. The minor formed on rows $\{1,3,5\}$ and on columns $\{2,4,6\}$ is:
$$
\left|\begin{array}{ccc}
0&b&d\\
s&0&x\\
u&y&0
\end{array}\right|=bx+dy>0.
$$
Thus we have found a $3 \times 3$ nonzero minor and so $M_{\mathcal{Z}}(G) \geq 3$. As $Z_{\pm}(G)=3$ we deduce that in fact $M_{\mathcal{Z}}(G)=3$.

\begin{figure}[h]
\begin{center}$
\begin{array}{cc}
\includegraphics[width=2.5in]{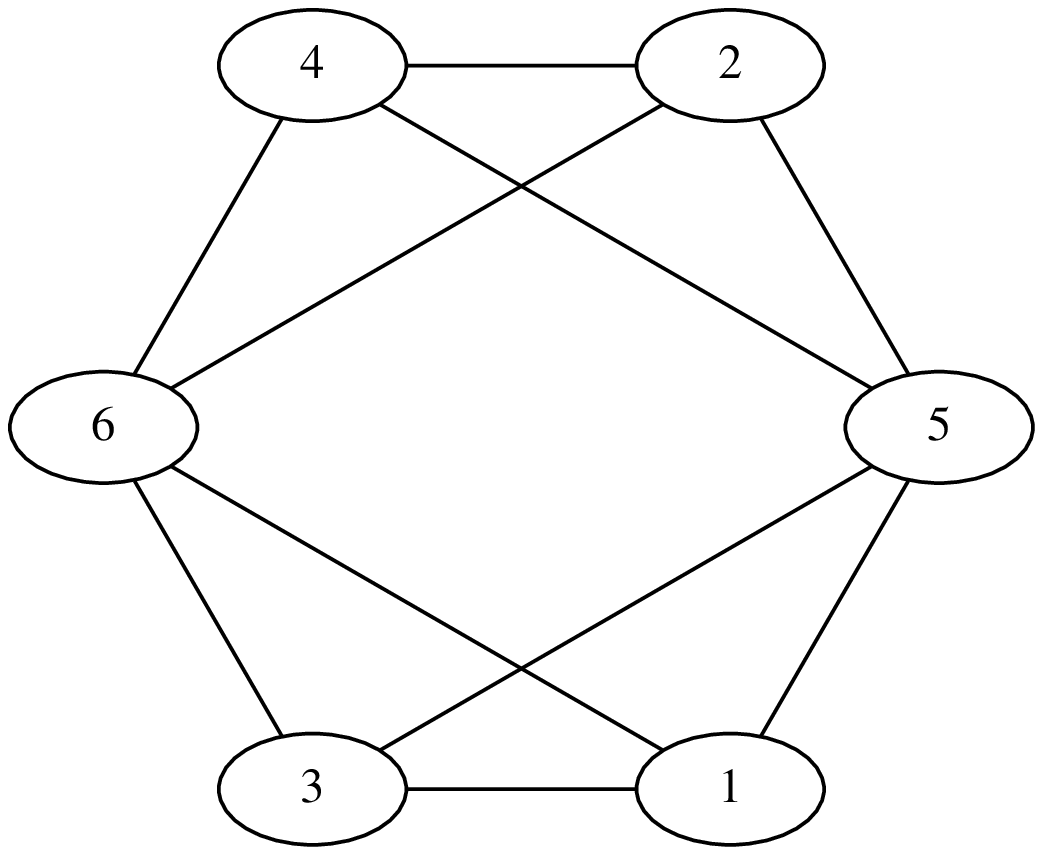} &
\includegraphics[width=2.5in]{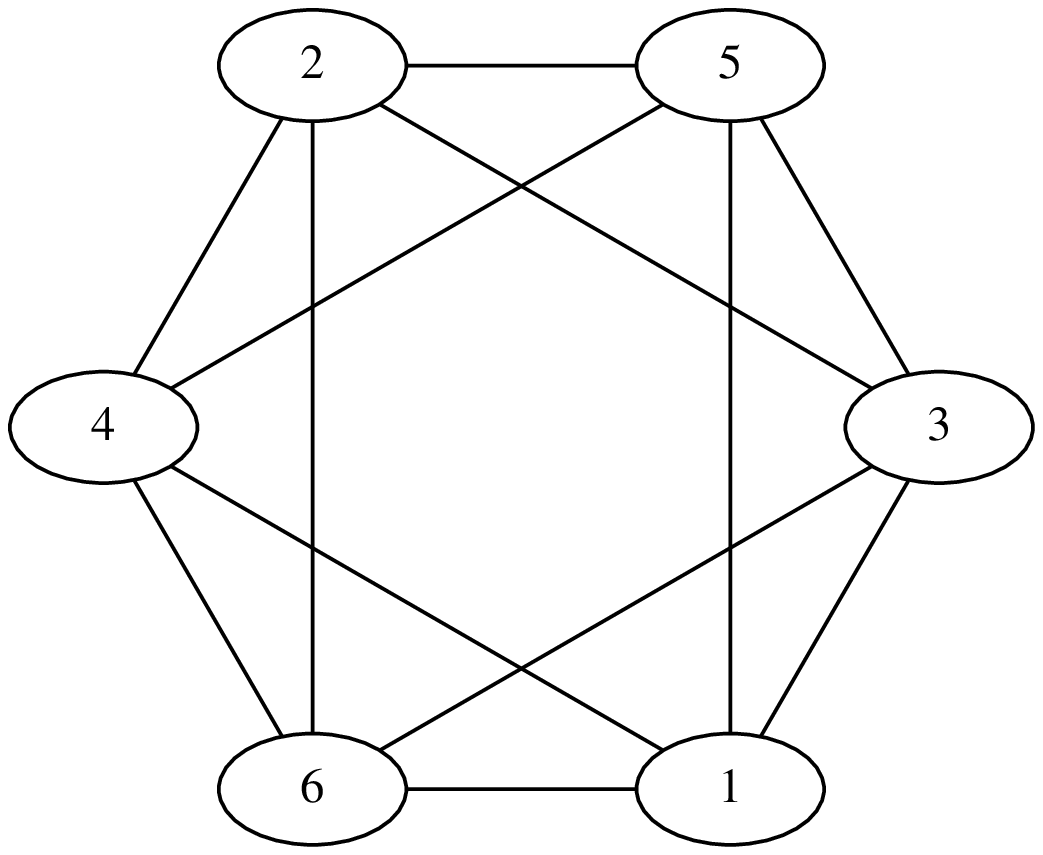}
\end{array}$
\end{center}
\caption{The smallest graphs for which $Z_{\pm}<Z$}\label{fig:min}
\end{figure}

The graph we have just considered is in fact a well-known one in a slight disguise, for it is isomorphic to the line graph $L(K_{4})$ of the clique on four vertices. Using a brute-force implementation we find that if $G=L(K_{5})$ then $Z(G)=7$ and $Z_{\pm}(G)=6$. In general the classical zero forcing number of $L(K_{n})$ is given by:
\begin{thm}\cite{EroKanYi12,AIM_zeroforce08}
For any $n \geq 4$, it holds that $$M(L(K_{n}))=Z(L(K_{n}))=\binom{n}{2}-(n-2).$$
\end{thm}

In the next subsection we are going to bound $M_{\mathcal{Z}}(L(K_{n})$ from above by computing $Z_{\pm}$. But before we do that we would like to find a lower bound that will coincide with the value of $Z_{\pm}$. 

\subsection{Line graphs of cliques}  
Recall that a \emph{clique cover} of a graph $G$ is a set $C_{1},C_{2},\ldots,C_{m}$ of cliques in $G$ such that every edge of $G$ belongs to at least one $C_{i}$. The minimum cardinality of a clique cover is called the \emph{clique-cover number} and denoted \emph{$cc(G)$}.

\begin{thm}\label{thm:cc}
For any graph $G$ on $n$ vertices it holds that $$M_{\mathcal{Z}}(G) \geq n-cc(G).$$
\end{thm} 
\begin{proof}
Let $m=cc(G)$ and let $C_{1},C_{2},\ldots,C_{m}$ be a minimum clique cover of $G$. For every $1 \leq i \leq m$ let $v_{i}$ be the indicator vector of $C_{i}$, i.e. the entries of $v_{i}$ corresponding to vertices in $C_{i}$ are equal to $1$ and all the other entries are zero. Now define $$A=\sum_{i=1}^{r}{v_{i}v_{i}^{T}}.$$

Clearly, the graph of $A$ is $G$ and $r(A) \leq cc(G)$. Therefore $M(G) \geq n-r(A) \geq n-cc(G)$.
\end{proof}

\begin{thm}\label{thm:lkn}
For $n \geq 6$, let $G=L(K_{n})$. Then it holds that $$M_{\mathcal{Z}}(G)=Z_{\pm}(G)=\binom{n}{2}-n.$$
\end{thm}
\begin{proof}
%First, we would like to find a signed zero forcing set of size $\binom{n}{2}-n$ in $G$. 
The vertices of $G$ can be labelled as $(i,j)$ with $i,j \in \{1,2,\ldots,n\}$ and $i<j$. Define $$S=V(G) \setminus (A \cup B),$$
where
$$
A=\{(i,n)|1 \leq i \leq n-3\}, $$ $$B=\{(n-2,n-1),(n-2,n),(n-1,n)\}.$$

We claim that $S$ is a signed zero forcing set for $G$. To see this, let us start the game by colouring all vertices in $S$ black. We first seed vertex $(1,n)$ with a $+$ and then apply Rule \ref{rul:new}(c) to vertex $(1,2)$ which allows us to mark $(2,n)$ with $-$.
%apply Rule \ref{rul:new}(d), the seeding clause, to vertex $(1,2)$ which has exactly two white neighbours: $(1,n)$ and $(2,n)$. 
%Without loss of generality we mark the vertex $(1,n)$ with a $+$ and the vertex $(2,n)$ with a $-$. 
Now apply Rule \ref{rul:new}(c) to all vertices of the form $(1,j),j=3,4,\ldots,n-3$. Each one of them has two white neighbours: $(1,n)$ and $(j,n)$ and since $(1,n)$ is marked with $+$, we can mark $(j,n)$ with $-$.

Next we apply Rule \ref{rul:new}(b) to vertex $(2,3)$: its two white neighbours are $(2,n)$ and $(2,3)$, both marked $-$. Therefore we can blacken both of these two vertices. At this stage vertex $(1,2)$ has only one white neighbour, that is $(1,n)$ and therefore \ref{rul:new}(a), the classical clause, allows us to blacken $(1,n)$. Similarly, applying Rule \ref{rul:new}(a) to vertex $(2,j)$ for $4 \leq j \leq n-3$ allows us to blacken the remaining white vertices of $A$. 

At the end of this process, we are left with only three white vertices - the vertices of $B$. The game can now be finished by seeding again from vertex $(n-3,n-2)$ which marks $(n-2,n-1)$ with $+$ and $(n-2,n)$ with $-$. Applying Rule \ref{rul:new}(c) to $(n-3,n-1)$ we can mark $(n-1,n)$ with $-$. Then Rule \ref{rul:new}(b) to $(n-2,n)$ allows us to blacken $(n-2,n-1)$ and $(n-2,n)$. Thus only $(n-2,n)$ is left white and there is any number of possible ways to deliver the \emph{coup de grace} to it. We have thus proved that
$$
Z_{\pm}(G) \leq |S|=\binom{n}{2}-n.
$$

On the other hand, if $H$ is any graph on $n$ vertices and $G=L(H)$ is its line graph, then $cc(G) \leq n$ since $G$ is clearly covered by the set of $n$ cliques of the form $C_{v}=\{e \in E(G)| v \in e\}$ for $v \in V(G)$. Therefore by Theorem \ref{thm:cc} we have:
$$
M_{\mathcal{Z}}(G) \geq \binom{n}{2}-n.
$$
An appeal to Theorem \ref{thm:main} finishes the proof.
\end{proof}

%%%\begin{rem}
%%%The bound of Theorem \ref{thm:cc} is not always attained as it happened to be in the proof of Theorem \ref{thm:lkn}. To see this, let $G=L(C_{4})$ be the 
%%%\end{rem}
 
%\section{Line graphs of cliques}

\section{Concluding remarks}\label{sec:remarks}
%In this section we offer a number of ideas for possible extensions of the method described in the paper. These ideas are tentative and further work needs to be done to see if they can really lead to tangible improvements.

\begin{enumerate}

\item
%Consider now again the situation when the game has terminated with vertex sets denoted $B,W_{+},W_{-},W_{*}$ as before, but now $B \neq V(G)$ and $W_{*}= \emptyset$ (we can also arrive at this point by multiple branching: if $|W_{*}|=k$, we can create $2^{k}$ branches, each corresponding to a possible assignment of signs to the vertices in $W_{*}$). 

%In the general case, there is nothing we can say at this point. 

Suppose now that the derived coloring of a set $S$ is not all black, but all white vertices have sign markers. If we somehow know in addition that $0$ is a low eigenvalue of $G$, we might be able to rule out the putative configuration of signs of the entries of $x \in \ker{A}$ resulting from our game by applying a suitable \emph{nodal domain theorem} (cf. \cite{DavGlaLeySta01}). 

This idea has been used in \cite[Section 5.3]{Mythesis} for the Colin de Verdi\`{e}re number $\mu(G)$ (in which case $0$ is stipulated to be the second lowest eigenvalue of $G$), with the help of a nodal domain-type result for $\mu$ due to van der Holst, Lov\'{a}sz and Schrijver.

\begin{qstn}
Is it possible to find a non-trivial class $\Gamma$ of graphs for which $0$ is guaranteed to be a low eigenvalue of any matrix in $\mathcal{Z}$ whose graph is $G \in \Gamma$? 
\end{qstn}

\item
Like the classical variety, our new game requires the sign pattern to have a decent helping of zeros to work well. The development of a method that works well for so-called \emph{full sign patterns} (i.e. without zeros) is still an open question. To try to reduce it to the sparse case, we pose a question:

\emph{Is it  possible to find rank-preserving transformations of zero patterns that increase the number of zeros?}

We remark that class of transformations which preserve sign-nonsingular zero patterns (\emph{i.e} patterns all of whose corresponding matrices are nonsingular) has been studied in \cite{BeaYe95}.

%\item
%The classical zero forcing game is deterministic and can be programmed on a computer to obtain the derived coloring of any given subset of vertices. Exhaustive search over all sets of vertices of sizes $1,2,3,\ldots$ can be then performed, until a subset of size $k$ is found whose derived coloring includes all vertices. Clearly $k=Z(G)$. This naive brute-force algorithm can be used to compute $Z(G)$ for graphs of a moderate size. Using a similar approach it is also possible to find $Z_{\pm}(G)$ for a moderately-sized $G$. 
%and so we are able to compare the two quantities.

\item
It seems that our method can be profitably applied to obtain bounds on the nullity of signed graphs, a problem which has been recently introduced and studied in \cite{AraHalLiHol13}.
\item
%We restricted our attention in this paper to patterns with fixed periphery
%It is conceivable to defin
%It is possible to explore 
To handle patterns with unspecified off-diagonal entries, it is possible to define a further variant of the signed forcing game in which edges are also marked with $+$ and $-$ markers, in a similar way to the way we marked vertices. We leave the exploration of this variant to future efforts.
\end{enumerate}

\section{Acknowledgments}
We are grateful to Leslie Hogben for interesting discussions about zero forcing and to the two anonymous referees for comments which have greatly enchanced both the substance and the presentation of the paper. Graphviz was used to draw the figures.

\bibliographystyle{abbrv}
\bibliography{cdv,nuim}

\end{document}